\documentclass{amsart}
\usepackage{graphicx}
\newtheorem{teo} {Theorem} 
\newtheorem{lem} {Lemma}
\newtheorem{prop} {Proposition} 
\newtheorem{claim} {Claim}

\begin{document}

\title{Ergodicity and annular homeomorphisms of the torus}






\author {Renato B. Bortolatto and Fabio A. Tal}
\thanks{Instituto de Matem\'{a}tica e estat\'{i}stica. Universidade de S\~{a}o Paulo (IME-USP). 
		First author supported by CNPq - Brasil.
		Second author supported by CNPq-Brasil, Proc. \#304360/2005-8.}


\begin {abstract}

Let $f: \mathbb{T}^2 \to \mathbb{T}^2$ be  a homeomorphism homotopic to the identity 
and $F: \mathbb{R}^2 \to \mathbb{R}^2$ a lift of $f$ such that the rotation set $\rho(F)$ is a 
line segment of rational slope containing a point in $\mathbb{Q}^2$.  
We prove that if $f$ is ergodic with respect to the Lebesgue measure on the torus 
and the average rotation vector (with respect to same measure)
does not belong to $\mathbb{Q}^2$ then some 
power of $f$ is an annular homeomorphism.

\end {abstract}


\maketitle
 
\section{Introduction}


The rotation set is a well known conjugation invariant for homeomorphisms of the torus that are 
homotopic to the identity. Inspired by the rotation number of Poincar\'{e} we can start with such homeomorphism 
$f$, fix a lift $F$ of $f$ acting on the plane and define the rotation set $\rho(F)$ as the set of 
accumulation points of the sequences 
\begin{displaymath}
\bigg\{ \frac{F^{n_i} (x_{i}) - x_{i}}{n_i} \bigg\}_{i \in \mathbb{N}}
\end{displaymath}
where $x_{i} \in \mathbb{R}^2$ for all $i \in \mathbb{N}$ and $n_i \in \mathbb{N}$ is such that 
$n_i \xrightarrow{i \to + \infty} + \infty$. 

For a homeomorphism $f$ on the torus homotopic to the identity the limit 
\begin{displaymath}
\lim_{n \to \infty} \frac{F^{n} (x) - x}{n} ,
\end{displaymath}
when it exists, is denoted by $\rho_p (F,x)$ and called pointwise rotation vector. We note that, 
if $\pi: \mathbb{R}^2 \to \mathbb{T}^2$ is the covering map, then $\rho_p (F, x) = \rho_p (F, y)$ 
whenever $\pi(x) = \pi(y)$. 
In contrast with the rotation number of Poincar\'{e} (for orientation preserving homeomorphisms of the 
circle) it does not need to exist for every $x \in \mathbb{R}^2$  
and even when it does exist for every $x \in \mathbb{R}^2$ the limit can still depend on $x$.  
Furthermore, $\cup_{x \in \mathbb{R}^2} \rho_p (F,x)$ is not necessarily equal to $\rho(F)$ 
as the later can be shown to be convex, unlike the former. 
For this reason, a vector $v \in \rho(F)$ for which there is $x$ satisfying $\rho_p(F, x) = v$ is said to be 
realized by $\pi(x)$. A source for this and other results  
is the seminal paper of Misiurewicz and Ziemian \cite{MZ}. 

The question of which dynamical properties of $f$ can be deduced from $\rho(F)$ is somewhat 
well understood, specially when $\rho(F)$ has non-empty interior. 
For example, when $\rho(F)$ has non-empty interior it's possible to 
prove that $f$ has strictly positive topological entropy \cite{LM}. 
It's also known that for each $(p/s,q/s) \in int(\rho(F))$ 
with $p, q, s \in \mathbb{Z}$ there exists $x \in \mathbb{R}^2$ such that 
$F^s (x) - x = (p, q)$ (see \cite{F2}). There exists a similar version 
when $\rho(F)$ is a non-degenerated segment and $f$ preserves area (see \cite{F1}). 

Our work here is focused on better understanding of the 
dynamical properties of $f$ when the rotation set has empty interior, particularly when the homeomorphism 
preserves the natural area measure $\lambda$ on $\mathbb{T}^2$, which we will refer to as  the 
Lebesgue measure on the $2$-torus or simply ``area".  Whenever $\lambda$ is invariant by $F$  one can 
define, as in \cite{MZ}, the average rotation vector with respect to $\lambda$ as 
\begin{displaymath}
\rho_{\lambda}(F) = \int_{x \in \mathbb{T}^2} (F(\pi^{-1}(x)) - \pi^{-1}(x)) d\lambda 
\end{displaymath}
We are particularly interested in understanding when the rotation set of $F$ 
prevents $f$ from having a strictly toral behavior (as defined in \cite{KT1}). 
In this note we prove the following theorem.

\begin{teo}
Let $f$ be a homeomorphism of $\mathbb{T}^2$ homotopic to the identity and 
$F$ be a lift of $f$ such that $\rho(F)$ is a non-degenerated 
line segment with rational slope that intercepts $\mathbb{Q}^2$. 
If the Lebesgue measure on the torus is $F$-invariant and ergodic and 
$\rho_{\lambda} (F)$ does not belong to $\mathbb{Q}^2$ then some 
power of $f$ is an annular homeomorphism. 
\end{teo}

Here, after \cite{KT1}, by annular homeomorphism we mean that there exists $M > 0$, 
$v \in (\mathbb{Z}^2 \setminus \{(0,0)\})$ and a lift $G$ of $f$ such that 
\begin{displaymath} 
- M \leq \bigg\langle G^n(x) - x, \frac{v}{\|v\|} \bigg\rangle \leq M, \quad
 \forall x \in \mathbb{R}^2, \forall n \in \mathbb{Z}
\end{displaymath} 
were $\| . \|$ denotes the Euclidean norm for $\mathbb{R}^2$. 

In \cite{D1} it is shown that if $\rho(F)$ is a (non-degenerated) line segment and there exists 
$p \in \mathbb{Q}^2 \cap \rho(F)$ which is not realized by a periodic orbit then some 
power of $F$ is annular. In Theorem $1$ we require that the Lebesgue 
measure is ergodic and $F$-invariant, so \cite{F1} ensures that any point in 
$\mathbb{Q}^2 \cap \rho(F)$ is realized by a periodic orbit. 

We emphasize that in \cite{KT2} the case where $\rho(F) = \{(0,0)\}$ 
is studied and it is shown that $f$ needs not to be annular, hence the hypothesis on the non-degeneracy of 
$\rho(F)$ into a point cannot be removed. 

It is not simple to exhibit an explicit example of homeomorphism that satisfies the hypothesis of Theorem $1$, 
so let us argue that these hypothesis are in fact very 
common. Firstly, let $F (x, y) = (x, y + \psi(x))$ where $\psi$ is any $1$-periodic, continuous function such that 
$\int_0^1 \psi(s) ds = \sqrt{2}$, and $\psi(0)=0$. 
Note that $F$ is a lift of a homeomorphism $f$ of the $2$-torus that is homotopic to the identity and preserves area. 
Additionally, it's easy to see that $\rho(F)$ is a non-degenerated line segment and that $\rho_{\lambda}(F) = (0, \sqrt{2})$.   

Now, let $\mathcal{H}_0 (\mathbb{T}^2, \lambda)$ be
the set of homeomorphisms of the $2$-torus which are homotopic to the identity, preserve the Lebesgue measure 
and, for a given $v \in \mathbb{R}^2$, let $\mathcal{H}_0^v (\mathbb{T}^2, \lambda)$ be the subset of 
$\mathcal{H}_0 (\mathbb{T}^2, \lambda)$ 
of homeomorphisms for which there exists a lift such that the average rotation vector with respect to $\lambda$ is $v$. 
The main result of \cite{AP}, an extension of the celebrated Oxtoby-Ulam Theorem, implies that ergodicity of the 
Lebesgue measure is a typical (dense $G_{\delta}$) property in both these spaces. 

The stated result, however, is still not sufficient for our purposes, but a careful look
at the proofs of both Proposition $3$ and Theorem $2$ of \cite{AP} shows that they prove
something more. In fact, if $I^2 = [0, 1] \times [0,1]$ and $\pi$ is the described map $\phi$ in Proposition $1$ of
\cite{AP}, their technique show that for every neighborhood $V$ of a given $f \in \mathcal{H}_0^v (\mathbb{T}^2, \lambda)$, 
there exists an homeomorphism $h$ such that, if $g = hf$ then $g$ is both
ergodic and belongs to $V$ and such that $h$ pointwise fixes any point in $\pi(\partial I^2)$. 
In particular, for the $F$ we took above we have that 
$F$ fixes pointwise the line $\{x \in \mathbb{R}^2 \; | \; (x)_1 = 0\}$. 
We can then find $g$ ergodic and with a lift $G$ having the 
same average rotation vector as $F$, and such that $F$ and $G$ coincide on $\partial I^2$. 
Then one can show that $\rho(G) = \{0\} \times [a, b]$, with $a \leq 0$ and $b \geq \sqrt{2}$, 
as $G$ has fixed points and $\rho_{\lambda}(G) = \rho_{\lambda}(F ) = (0, \sqrt{2})$. 

The technique we use in the proof of Theorem $1$ is to study the sets 
$B_{0}$, $B_{\pi }$, $\omega(B_{0})$ and $\omega(B_{\pi})$ as defined in 
\cite{AT1, AT2}. 
The presence of these sets, which we'll describe in the next section, 
have many important dynamical consequences that 
have proven useful in obtaining a number of results 
(see for instance \cite{T1}, \cite{ATG}, \cite{D1}). 
Many ideas used in our proof here follow from \cite{T1}. 


\section{Preliminaries}  


In this work we consider $\mathbb{T}^2 = \mathbb{R}^2 / \mathbb{Z}^2$ to be the flat $2$-torus, and $\lambda$ its area  
measure, which we call the Lebesgue measure. The covering projection from the universal cover 
$\mathbb{R}^2$ to $\mathbb{T}^2$ is denoted by $\pi$. Given a point $x \in \mathbb{R}^2$ and the canonical basis, 
we denote by $(x)_1$ (respectively $(x)_2$) its first (resp. second) coordinate of $x$. 
As noted before, we denote by $\mathcal{H}_0(\mathbb{T}^2, \lambda)$ 
the set of area preserving homeomorphisms of $\mathbb{T}^2$ which are homotopic to the identity.

Let $S \subseteq \mathbb{R}^2$ with $S \neq \emptyset$. 
We will say that $S$ is unbounded rightward if $\sup_{x \in S} (x)_1 = + \infty$, and we will say 
that $S$ is unbounded leftward if $\inf_{x \in S} (x)_1 = - \infty$. If $S$ is either unbounded leftward or rightward we say that 
$S$ is horizontally unbounded. Otherwise $S$ is horizontally bounded, in which case there are real numbers $a, b$ such 
 that $S$ is contained in 
 $$([a, +\infty[ \times \mathbb{R}) \cap (] - \infty, b] \times \mathbb{R}) = [a, b] \times \mathbb{R}$$ 
In this case we call any real number greater than $|b - a|$ a horizontal bound for $S$. 
Likewise, a non-empty set $S \subseteq \mathbb{R}^2$ will be called unbounded upward if  
$\sup_{x \in S} (x)_2 = + \infty$ and unbounded downward if $\sup_{x \in S} (x)_2 = - \infty$. 
A set $S \subseteq \mathbb{R}^2$ will be called vertically unbounded if it is unbounded upward or downward. Again, 
in this case there are real numbers $a, b$ such that $S$ is contained in $\mathbb{R} \times [a, b]$. 

To fix the terminology, given a curve $\gamma:[0,1]\to \mathbb{R}^2$, we denote by $[\gamma]$ its image. 
By $\lfloor . \rfloor : \mathbb{R} \to \mathbb{Z}$  we'll mean the usual floor function. 
Finally, we'll say that a set $S \subset \mathbb{R}^2$ separates two sets $L, R \subset \mathbb{R}^2$ if $L$ and 
 $R$ are in distinct connected components of the complement $S^C$.   
 
The following two results are relevant in our proofs. The first is a theorem 
of J. Franks we mentioned in the introduction and the second can be deduced from 
a theorem of G. Atkinson. 

\begin{lem}[\cite{F1}]
Suppose that $f \in \mathcal{H}_0(\mathbb{T}^2, \lambda)$ and $F$ is a lift such that $\rho(F)$ 
is a non-degenerated line segment. Then for every $(p/s,q/s) \in \rho(F)$ with $p, q, s \in \mathbb{Z}$ there exists 
$x \in \mathbb{R}^2$ such that $F^s (x) - x = (p, q)$.
\end{lem}  

\begin{lem}[\cite{A1}]
Let $M$ be a compact manifold and $f: M \to M$ be continuous. 
Let $\mu$ be a Borelian, ergodic probability measure. 
Let $g: M \to \mathbb{R}$ be continuous and satisfy
\begin{displaymath}
\int_M g d\mu = 0
\end{displaymath}
If $A$ is a Borel set and $\mu(A) > 0$ then for $\mu$-a.e. $x \in A$ there is a sequence 
$n_k \xrightarrow{k \to \infty} \infty$ such that 
\begin{displaymath}
f^{n_k} (x) \xrightarrow{k \to \infty} x \quad 
\text{and } \quad \sum_{i = 0}^{n_k - 1} g(f^{i} (x)) \xrightarrow{k \to \infty} 0
\end{displaymath}
\end{lem} 

Our main result is a direct consequence of the following statement

\begin{teo}
Let $f \in \mathcal{H}_0(\mathbb{T}^2, \lambda)$ and let $F: \mathbb{R}^2 \to \mathbb{R}^2$ be a lift of $f$ such that 
$\rho(F) = \{0\} \times [a,b]$, with $a \leq -1 < 1 \leq b$.  
Suppose $f$ is ergodic with respect to the Lebesgue measure $\lambda$ and that
$\rho_{\lambda}(F)$ is of the form $(0, \alpha)$ for some $\alpha \in \mathbb{R} \setminus \mathbb{Q}$. 
Then there is $M > 0$ such that $| ( F^n(x) -x )_1 | \leq M$ for all $n \in \mathbb{Z}, \; x \in \mathbb{R}$ .
\end{teo} 

Let us show that Theorem $1$ follows from Theorem $2$. 
Let $f,g \in \mathcal{H}_0(\mathbb{T}^2, \lambda)$ and let $F$ and $G$ be lifts of $f$ and $g$, respectively.  
Assume $\rho(F) = \{0\} \times [a,b] $ and $\rho(G)$
is a non-degenerated line segment with rational slope that intercepts $\mathbb{Q}^2$.  
Then there is an invertible map $A \in GL(2, \mathbb{Z})$ 
such that $A\rho(G)$ is a line segment of the form $\{ \frac{p}{q} \} \times [a,b]$. 
Using lemma $2.4$ in \cite{KK} we see $A$ as a change of coordinates in 
the torus (and as such no other property is destroyed by $A$). 
We can then assume, taking $h = (Ag)^q$ and $H = (AG)^q - (p, 0)$, that 
$\rho(H) = \{ 0\} \times [a,b]$ (see \cite{MZ}). 
Taking a power $n$ of $h$ such that the length of $\rho(H^n)$ is greater than three 
and changing the lift we can assume that $a \leq -1 < 1 \leq b$.  


\section{Definitions and first properties}


From now on, let $f$ be a homeomorphism of the $2$-torus isotopic to the identity, and let $F$ be a lift of $f$. 
Let $e_0 = (1,0)$ and $e_\pi = (-1, 0)$. We define 
\begin{displaymath}
V_{0}^{+} := \{ x \in \mathbb{R}^2 | \langle x,e_0 \rangle \geq 0 \} =  \{ x \in \mathbb{R}^2 | (x)_1 \geq 0 \}
\end{displaymath}
and
\begin{displaymath}
V_{\pi}^{+}:= \{ x \in \mathbb{R}^2 | \langle x,e_\pi \rangle \geq 0 \} =  \{ x \in \mathbb{R}^2 | (x)_1 \leq 0 \}
\end{displaymath}

Consider $\mathbb{R}^2 \cup \{ \infty\} \sim S^2$ the one-point compactification of 
$\mathbb{R}^2$ and $\widehat{F}$ the homeomorphism 
induced by $F$ on $S^2$ fixing the point in the infinity. The sets 
$\widehat{V_{0}^{+}} := V_{0}^{+} \cup \{\infty\}$ and 
$\widehat{V_{\pi}^{+}} := V_{\pi}^{+} \cup \{\infty\}$ correspond to $V_{0}^{+}$ and $V_{\pi}^{+}$, 
respectively, on $S^2$.

Let $\widehat{B_{0}}$  be the connected component of 
\begin{displaymath}
\bigcap_{n \leq 0} \widehat{F}^n (\widehat{V_{0}^{+}}) 
\end{displaymath}
that contains the point at the infinity.  Let $\widehat{B_{\pi}}$ be the connected component of 
\begin{displaymath}
\bigcap_{n \leq 0} \widehat{F}^n (\widehat{V_{\pi}^{+}}) 
\end{displaymath}
that contains the point at the infinity.

Define the sets $B_{0}$ and $B_{\pi}$ in $\mathbb{R}^2$ that correspond, respectively, 
to the sets $\widehat{B_{0}}$ and $\widehat{B_{\pi}}$ on $S^2$. To avoid confusion, in this 
work we'll use the notation $B(x; \varepsilon)$ for the ball of  center $x$ (either in $\mathbb{T}^2$ or in $\mathbb{R}^2$) 
and radius $\varepsilon \geq 0$.

We'll need the following result. 

\begin{lem} \label{nonempty}
Let $F$ be a lift of a homeomorphism $f$ homotopic to the identity with 
$(0,0)\in \rho(F)$.
Then $B_{0}$ and $B_{\pi}$ are not empty. 
\end{lem}

\begin{proof}
This follows the ideas in \cite{ATG}, but since the context is somewhat different, we  include the proof for completeness.
We will show that $B_{0}$ is nonempty, the other case is analogous.  

First, assume that, for every $M>0,$ there exists $x\in \mathbb{R}^2$ and $n>0$ such that $(F^n(x)-x)_1\geq M + 1$. 
In this case, for every positive $M$ we can find a positive integer $n(M)$ such that $F^{n(M)}(V_{\pi}^+)$ intersects 
$V_0^{+}+(M,0)$. We claim that, if $M>1,$ this implies $F^{n(M)}(\partial V_0^+)$ intersects 
$\partial V_0^+ + (M, 0)$. If this was not the case, 
since $F^{n(M)}(V_{\pi}^{+})$ intersects but does not contain $V_0^{+}+(M,0)$,  it would follow that 
$F^{n(M)}(V_0^{+})\subset \left(V_0^{+}+(M,0)\right) \subset \left(V_0^{+}+(1,0)\right).$ 
But these inclusions imply that, for every $x\in V_0^{+},\, \liminf_{i\to\infty}\frac{(F^i(x)-x)_1}{i}\geq \frac{1}{n(M)}$, 
and since $\pi(V_0^{+})=\mathbb{T}^2,$ we would have $\rho(F)\subset V_0^{+}+(\frac{1}{n(M)},0)$, a contradiction.
Therefore, for every $M>1$  there exists $n(M)$ such that $F^{n(M)}(\partial V_0^+ )\cap (\partial V_0^+ + (M, 0))$ 
is not empty. The result now follows exactly like lemma 1 of \cite{T1}.

Assume now that there exists $M>0$ such that, for all $x$ and all positive integers $i$, $(F^{i}(x)-x)_1<M$. 
Let $K=\bigcup_{i>0}F^{-i}(V_0^{+}+(M,0))$ which is a connected unbounded set, satisfying $F^{-1}(K)\subset K$ and 
$K\subset V_0^{+}.$ Note also that $K=K+(0,1)$.  Now, if $\widehat K$ is the corresponding set  in $S^2$, then 
$\widehat{F}^{-i}(K)$ is a nested sequence of connected compact sets, all of which containing the infinity. Let 
$\widehat{K_{\infty}}=\bigcap_{i=0}^{\infty}\widehat{F}^{-i}(\widehat{K})$ be their intersection, and let 
$K_{\infty}$ the corresponding set in $\mathbb{R}^2$. 

We claim $K_{\infty}$ is not empty. Otherwise, by compactness, there would be a first integer $n>0$ such that $F^{-n}(K)$ 
does not intersect the fundamental domain $[M, M+1]\times [0,1].$ Since $F^{-n}(K)$ is invariant by integer vertical 
translations, it must also be disjoint from the infinite strip $[M,M+1]\times \mathbb{R}$, and as $F^{-n}(K)$ is connected, it 
would be contained in $V_0^{+}+(M+1,0)$. In particular, $F^{-n}(V_0^{+}+(M,0))\subset V_0^{+}+(M+1,0).$ This implies that, 
for all $x \in V_0^{+}+(M,0)$, $\liminf_{i\to\infty} \frac{(F^{-i}(x)-x)_1}{i} \geq \frac{1}{n}$ which again contradicts 
$\{(0,0)\} \in \rho(F).$

But then, since $\widehat{K_{\infty}}$ is connected and contains the infinity, every connected component of 
$K_{\infty}$ is unbounded, and since 
$$K_{\infty}=\bigcap_{l=0}^{\infty}\left(\bigcup_{i=l}^{\infty}F^{-i}(V_0^{+}+(M,0))\right)$$
is a $F$-invariant set and $K_\infty\subset V_0^{+},$ it follows that $K_\infty\subset B_0$.
\end{proof}

Note that, since $\widehat{V_{0}^{+}}, \widehat{V_{\pi}^{+}}$ are closed and 
$\widehat{F}$ is a homeomorphism, the sets $\widehat{B_0}, \widehat{B_{\pi }}$ are closed and 
therefore the sets $B_{0}, B_{\pi}$ are also closed. The set  $B_{0}$ can be seen as 
the union of all connected closed, unbounded sets $C$ of $\mathbb{R}^2$ that satisfy 
\begin{equation} \tag{*}
(F^n(x))_1 \geq 0 \;  \forall n \in \mathbb{N}^+ 
\end{equation}
for all $x \in C$. There is a analogous characterization for $B_{\pi}$. 

We now define the $\omega$-limit of $B_{0}$ as usual by 
\begin{displaymath}
\omega(B_{0}) := \bigcap_{i =1}^{\infty} \overline{\bigcup_{j = i}^{\infty} F^j (B_{0})} = \bigcap_{i = 1}^{\infty} F^{i} (B_{0})
\end{displaymath}

The sets $\omega(B_{0})$ and $\omega(B_{\pi})$ 
are closed and all of its connected components are unbounded (see \cite{T1} proposition $1$ items $2,3$ and 
proposition $2$ items $2,3$). 
Note that $\omega(B_{0}) \subseteq B_{0}$ 
(since $F(B_{0}) \subseteq B_{0}$) and that $\omega(B_{\pi}) \subseteq B_{\pi}$. 
It's easy to see that $\omega(B_{0})$ and $\omega(B_{\pi})$ are completely invariant, i.e., 
that $F^i(\omega(B_{0})) = \omega(B_{0})$ for all $i \in \mathbb{Z}$. We'll also need the following proposition 
(see \cite{T1} proposition $1$ items $3, 4$ and proposition $2$ items $4,5$). 
The equality is not covered in \cite{T1} but can be deduced from the same arguments. 

\begin{prop}  \label{vert inv}
The sets $B_0^C, B_{\pi }^C, \omega(B_{0})^C$ and $ \omega(B_{\pi})^C$  
satisfy the following properties. 
\begin{enumerate}
\item{Each of the sets $B_0^C, B_{\pi }^C, \omega(B_{0})^C$ and $ \omega(B_{\pi})^C$  has a single connected component. }
\item{If $(p, q) \in \mathbb{Z}^2$ with $p \geq 0$ then $B_{0} + (p, q) \subseteq B_{0}$ 
and  $B_{\pi} + (- p, q) \subseteq B_{\pi}$. Furthermore, 
$B_{0} + (0, q) = B_{0}$ and $B_{\pi} + (0, q) = B_{\pi}$ for all $q \in \mathbb{Z}$.}
\item{If $(p, q) \in \mathbb{Z}^2$ with $p \geq 0$ then $\omega(B_{0}) + (p, q) \subseteq \omega(B_{0})$ and  $\omega(B_{\pi}) + (-p, q) \subseteq \omega(B_{\pi})$. Furthermore, 
$\omega(B_{0}) + (0, q) = \omega(B_{0})$ and $\omega(B_{\pi}) + (0, q) = \omega(B_{\pi})$ 
for all $q \in \mathbb{Z}$.}
\end{enumerate}
\end{prop}

\begin{prop} 
If $\rho(F) = \{0\} \times [a, b]$ then $\omega(B_{0})$ and $\omega(B_{\pi})$ are both non-empty.
\end{prop}

\begin{proof}
By corollary $1$ in \cite{T1} if $\omega(B_{0})$ (respectively $\omega(B_{\pi})$) was empty 
we would have that $\rho(F) \cap int(V_{0}^{+}) \neq \emptyset$ 
(respectively  $\rho(F) \cap int(V_{\pi}^{+}) \neq \emptyset$), a contradiction. 
\end{proof}

With this information we can divide the proof in two proper cases, namely, either 
$\pi (\omega(B_{0})) \cap \pi(\omega(B_{\pi})) = \emptyset$ or 
$\pi (\omega(B_{0})) \cap \pi(\omega(B_{\pi})) \neq \emptyset$. 
In both cases if Theorem $2$ isn't true we'd obtain contradictions with properties of the sets 
$\omega(B_{0})$ and $\omega(B_{\pi})$. 
To achieve this we'll first prepare some statements on what happens when Theorem $2$ fails. 


\section{Proof of Theorem $2$, initial claims} 


The proof of Theorem $2$ will be done by contradiction so henceforth we assume that 
Theorem $2$ is not true. 
Define the set
\begin{displaymath}
A := (\omega(B_{0}) \cup (\omega(B_{\pi}) + (p, q)) )^C
\end{displaymath}
This set depends a priori of $p$ and $q$, but we have in general the following result. 

\begin{prop} \label{A nao vazio}
Let $M$ be a positive real number. 
Then there is $x_M \in \mathbb{R}^2$ and $n(x_M) \in \mathbb{Z}$ such that 
\begin{displaymath}
(x_M)_1 < - M \qquad \text{and} \qquad (F^{n(x_M)} (x_M))_1 \geq M 
\end{displaymath}
In particular if $(p, q) \in \mathbb{Z}^2$ is given we can take $M > 0$ such that $x_M \in A$. 
\end{prop}

\begin{proof}
By the contradiction hypothesis there exists $x \in \mathbb{R}^2$ and $n(x)$ such that $ |(F^{n(x)}(x) -x )_1|   \geq 2M + 1$. 
Assume that $(F^{n(x)}(x) - x )_1 > 0$ so that $|(F^{n(x)}(x) -x )_1| =  (F^{n(x)}(x) -x )_1$. 

Since $0 \leq (x)_1 - \lfloor (x)_1 \rfloor < 1$. The point 
$x_0 := x - (\lfloor (x)_1 \rfloor + M + 1, 0)$ satisfies 
\begin{displaymath}
- M - 1 \leq  (x_0)_1  < - M
\end{displaymath}

Since $(x_0)_1  < - M$ we know $x_0 \notin V_{0}^{+}$ so clearly $x_0 \notin \omega(B_{0})$ . 

But, since $(\lfloor (x)_1 \rfloor + M + 1, 0) \in \mathbb{Z}^2$, we know that 
\begin{displaymath}
 (F^{n(x)}(x_0) - x_0)_1  =  (F^{n(x)}(x) - x)_1  \geq 2M + 1 
\end{displaymath}
which, in turn, implies that 
\begin{displaymath}
 (F^{n(x)}(x_0))_1  \geq 2M + 1 + (x_0)_1 \geq 2M + 1 - M - 1 = M 
\end{displaymath}

Therefore $x_0$ is also not in $(\omega(B_{\pi}) + (M - 1, q))$ 
since one of its iterates has first coordinate greater of equal than $M$ (see (*) in the previous section 
and recall that $\omega(B_{\pi}) + (M - 1, q) \subseteq V_{\pi}^+ + (M - 1,q)$). 
Taking in particular $M = p + 1$ we conclude that $x_0 \notin (\omega(B_{\pi}) + (p, q))$ and   
therefore $x_0 \in A$.

For the proof in the case $(F^{n(x)}(x) - x )_1 < 0$ it's enough to define 
$y := F^{n(x)}(x) $ so that $(F^{-n(x)} (y) - y )_1 > 0$. 
\end{proof}

\begin{prop} \label{ilim horiz} 
All connected components of 
$\omega(B_{0})$ and $\omega(B_{\pi})$ are horizontally unbounded. 
\end{prop}

\begin{proof}
We'll show that all connected components of $\omega(B_{0})$ and 
$\omega(B_{\pi})$ cannot be vertically unbounded and horizontally bounded. 
In particular, since we know that these connected components are unbounded, 
we'll conclude that all connected components of $\omega(B_{0})$ and $\omega(B_{\pi})$
are horizontally unbounded (note that, since $\omega(B_{0}) \subseteq B_{0}$, any connected 
component of $\omega(B_{0})$ can only be unbounded rightward, just as  any connected 
component of $\omega(B_{\pi})$ can only be unbounded leftward).   

Suppose that a connected component $C$ of $\omega(B_{0})$ is vertically unbounded and horizontally bounded.  
Since $A$ is invariant by vertical translation,  for  $i \in \mathbb{Z}$,  
$C + (0, i)$ is also a connected component of $\omega(B_{0})$ and is also  
vertically unbounded and horizontally bounded. 

Clearly if $M \geq \sup_{x \in C} | (x)_1|$ then $\cup_{i = - \infty}^{+ \infty} (C + (0, i))$ 
separates the sets  
$R := \{ x \in \mathbb{R}^2 \; | \; (x)_1 \geq M  \}$ and 
$ L := \{ x \in \mathbb{R}^2 \; | \; (x)_1 \leq - M  \}$. 
We know that $\omega(B_{0})^C$ has a single connected component, and since $L\subset (V_0^{+})^C\subset \omega(B_{0})^C$,
it follows that $R\subset \omega(B_{0})$.

 But the previous proposition implies there is a point of $L\subset\omega(B_{0})^C$ that has one of its iterates 
with first coordinate greater or equal than
$M + 1$ and therefore belongs to $R$. 
But this is a contradiction since $\omega(B_{0})^C$ is completely invariant. 

The case where $R$ is completely contained in $\omega(B_{0})^C$ is analogous: 
proceeding as in the previous proposition  take $y := F^{n(x)} (x) $. 
The proof for the connected components of $\omega(B_{\pi})$ is also analogous. 
\end{proof}

We will now examine the two different possibilities, first the case where 
$\pi (\omega(B_{0})) \cap \pi(\omega(B_{\pi})) \neq \emptyset$ and second the case 
where $\pi (\omega(B_{0})) \cap \pi(\omega(B_{\pi})) = \emptyset$.


\section{The case where $\pi (\omega(B_{0})) \cap \pi(\omega(B_{\pi})) \neq \emptyset$ leads to a contradiction} 


In this section we prove that, since we're assuming that  Theorem $2$ is not true, 
we cannot have $\pi (\omega(B_{0})) \cap \pi(\omega(B_{\pi})) \neq \emptyset$. We start 
noticing that if $\pi( \omega(B_{0}) ) \; \cap \; \pi(\omega(B_{\pi})) \neq \emptyset$ then there 
are $(p_1, q_1), (p_2, q_2) \in \mathbb{Z}^2$ such that 
\begin{displaymath}
\big( \omega(B_{0}) + (p_1, q_1) \big) \cap \big( \omega(B_{\pi}) + (p_2, q_2) \big) \neq \emptyset
\end{displaymath}
Let $(p,q) = (p_2 - p_1, q_2 - q_1)$. By the hypothesis there is $z \in \mathbb{R}^2$ such that 
\begin{displaymath}
z \in \omega(B_{0}) \cap \big( \omega(B_{\pi}) + (p, q) \big)
\end{displaymath}
We'll need the following results. 

\begin{claim} 
If $O$ is a connected component of 
\begin{displaymath}
A = (\omega(B_{0}) \cup (\omega(B_{\pi}) + (p, q)) )^C
\end{displaymath}
then $O + (0, i)$ is a connected component of $A$ for all $i \in \mathbb{Z}$.
\end{claim}

\begin{proof}
By proposition $\ref{vert inv}$ we have
$A^C + (0, i) = A^C$  for all $i \in \mathbb{Z}$. 
This implies that $A + (0, i) = A$ for all $i \in \mathbb{Z}$, 
which proves the desired result.
\end{proof}

\begin{figure}[t]
\includegraphics[scale=0.75]{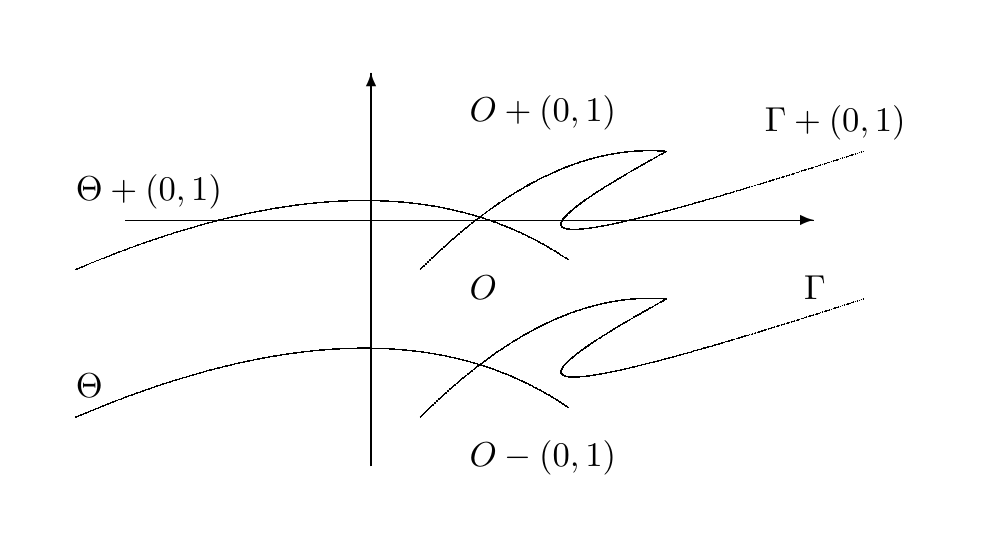}
\caption{$\Gamma$ and $\Gamma + (0,1)$ are connected components of $\omega(B_0) \subseteq V_0^+$ and 
$\Theta$ and $\Theta + (0,1)$ are connected components of $\omega(B_\pi) + (p,q) \subseteq V_{\pi}^+  + (p,q)$.   
The sets $O, O + (0,1)$ and $O - (0,1)$ are different connected components of $A$}
\end{figure}

This next claim is illustrated in Figure $1$. 

\begin{claim} \label{diffcc}
Let $O$ be a connected component of $A$ and let $x \in O$. 
Then $(x + (0,i)) \notin O$ for all $i \in \mathbb{Z}^*$. 
In particular, $O$ and $O + (0,i)$ are distinct connected component of $A$ for all 
$i \in \mathbb{Z}^*$.
\end{claim}

\begin{proof}
Define, for all $i \in \mathbb{Z}$,
\begin{displaymath}
y(i) := x + i(0, 1)
\end{displaymath}
Observe that $y(i) \in A$ for all $i \in \mathbb{Z}$. 
We'll show that every point $y(i)$ must be in a different connected component of $A$. 

Suppose this is not the case. In particular, there are $i_1 < i_2$ such that $y(i_1)$ and 
$y(i_2)$ are in the same connected component of $A$. 
Let $\gamma$ be a curve connecting $y(i_1)$ and $y(i_2)$ with $[\gamma] \subset A$. 
Then the curves $\gamma + (0, i)$ also have their images in $A$, for all $i \in \mathbb{Z}$. 

The set $\gamma := \cup_{k \in \mathbb{Z} } ([\gamma] + k(0, i_2 - i_1))$ is connected, 
vertically unbounded (both upward and downward), horizontally bounded and is contained in $A$. 
Note that $\gamma$ separates the sets
$R := \{ x \in \mathbb{R}^2 \; | \; (x)_1 > M  \}$ and
 $ L := \{ x \in \mathbb{R}^2 \; | \; (x)_1 < - M  \}$ for any $M$ greater than the horizontal bound for  $\gamma$. 

By proposition $\ref{ilim horiz}$, the connected components of $\omega(B_{0})$ and 
$\omega(B_{\pi})$ are horizontally unbounded. 
Let $z \in \omega(B_{0}) \cap (\omega(B_{\pi}) + (p, q))$. Let $\Gamma$ be the connected component of
 $\omega(B_{0})$ that contains $z$ and 
$\Theta$ the connected component of $\omega(B_{\pi}) + (p,q)$ that contains $z$. 
Since $\omega(B_{0}) \subseteq V_{0}^{+}$, $\omega(B_{\pi}) \subseteq V_{\pi}^{+}$ we conclude that 
$\Lambda := \Gamma \cup \Theta$ is connected, unbounded rightward, unbounded leftward
 and $\Lambda \cap A = \emptyset$, but this leads to a contradiction since $\Lambda \cap \beta \neq \emptyset$. 
\end{proof}

Since the Lebesgue measure on the torus $\lambda$ is ergodic with respect to $f$, we have that for $\lambda$-almost all $x \in \mathbb{T}^2$   
\begin{displaymath}
\lim_{n \to \infty } \frac{F^n \circ\pi^{-1} (x) - \pi^{-1}(x)}{n} = \int_{x \in \mathbb{T}^2} (F\circ\pi^{-1} (x) - \pi^{-1} (x)) d\lambda  = (0, \alpha)
\end{displaymath}
that is, for $\lambda$-almost every point the pointwise rotation vector exists and 
is of the form $(0, \alpha)$ for some $\alpha$ that we assumed irrational. 

Let $x \in A$ and let $O$ be the connected component of $A$ that contains $x$. 
Since $A$ is open there is $\varepsilon > 0$ such that $B(x; \varepsilon) \subseteq O$. 
Observe that,  for all $\varepsilon > 0$ we have $\lambda(\pi( B(x; \varepsilon) ) ) > 0$. 
Let $g := (F \circ\pi^{-1} - \pi^{-1})_1$. Then,  
by Atkinson's lemma, there is a sequence $n_j \xrightarrow{j \to \infty} \infty$ and $p \in \pi( B(x; \varepsilon))$ such that for 
$y \in \pi^{-1}(p)$ we have that $f^{n_j} (p) \xrightarrow{j \to \infty} p$ and 
\begin{displaymath}
(F^{n_j}( y)  - y)_1 \xrightarrow{j \to \infty} 0
\end{displaymath}
We can assume that $\rho_{p} (F, y)=(0,\alpha)$, since this holds $\lambda$-almost everywhere. 
Also note that for some $k_j \in \mathbb{Z}$ we have $F^{n_j}(y) \in (B(x; \varepsilon) + (0,k_j))$. 
Since $F^{n_j}$ is continuous and $A$ is $F$-invariant, $F$ permutes connected components of $A$. 
By the previous claim 
\begin{displaymath}
F^{n_j}(O) = O + (0, k_j)
\end{displaymath}
which implies that, for all $s \in \mathbb{Z}$,
\begin{equation} \label{sobe}
F^{s n_j}(O) = O + s(0, k_j)
\end{equation}

\begin{claim}
There are $j_1, j_{2} \in \mathbb{N}$ such that  
$\frac{k_{j_1}}{n_{j_1}} \neq \frac{k_{j_2}}{n_{j_2}}$. 
\end{claim}

\begin{proof}
Assume by contradiction that $\frac{k_{j_1}}{n_{j_1}} = \frac{k_{j_m}}{n_{j_m }}$ 
for all $m \in \mathbb{N}$. We then have 
\begin{displaymath}
k_{j_m} - \varepsilon  \leq (F^{n_{j_m}}(y) - y)_2 \leq k_{j_m} + \varepsilon
\end{displaymath}
Dividing by $n_{j_m}$ we can take the limit as $m \to \infty$ to see 
(since $\rho_{p} (F, y)$ exists) that 
\begin{displaymath}
\alpha = \lim_{m \to \infty} \frac{k_{j_m}}{n_{j_m}} = \frac{k_{j_1}}{n_{j_{1}}} 
\end{displaymath}
a contradiction, since $\alpha$ is irrational. 
\end{proof}

Substituting $n_{j_2}$ and $n_{j_1}$ in \eqref{sobe} we have
\begin{displaymath}
 O + n_{j_1}(0, k_{j_2}) = F^{n_{j_1} n_{j_2}}(O) = F^{n_{j_2} n_{j_1}}(O) = O + n_{j_2}(0, k_{j_1}) 
\end{displaymath}

But this leads to a contradiction: by claim $\ref{diffcc}$, since 
$\frac{k_{j_1}}{n_{j_1}} \neq \frac{k_{j_2}}{n_{j_2}}$, 
we must have that $O + n_{j_1}(0, k_{j_2}) \neq O + n_{j_2}(0, k_{j_1})$.
Therefore, we conclude that  the case 
$\pi (\omega(B_{0})) \cap \pi(\omega(B_{\pi})) \neq \emptyset$  
cannot be.


\section{The case where $\pi (\omega(B_{0})) \cap \pi(\omega(B_{\pi})) = \emptyset$  leads to a contradiction} 


Since the first possibility lead to a contradiction we examine now the remaining case. 
It's evident that if $\pi (\omega(B_{0})) \cap \pi(\omega(B_{\pi})) \neq \emptyset$, 
the distance between these sets is zero. 
Nevertheless, the same is still true if $\pi (\omega(B_{0})) \cap \pi(\omega(B_{\pi})) = \emptyset$. 

\begin{prop} \label{predense}
For every $z \in \mathbb{T}^2$ and all $\varepsilon > 0$, there is a connected set 
$K \subset \mathbb{R}^2$ 
vertically unbounded (both upward and downward), horizontally bounded and such that for all 
$y \in K$ there is $n(y) \leq 0$ such that
$\pi(F^{n(y)} (y) ) \in B( z; \varepsilon)$. 
\end{prop}

\begin{proof}
Define $O_0 = B(x; \varepsilon)$, where $x \in \pi^{-1}(z)$. 
As in the previous section we can use Atkinson's lemma to find
$n_j \in \mathbb{N}^*$, $k_j \in \mathbb{Z}$ and $y \in B(x; \varepsilon)$ such that 
$F^{n_j}(y) \in B(x; \varepsilon) + (0,k_j)$ and $\rho_p(F,y)=(0,\alpha)$.  
Arguing as is Claim $3$ we obtain $j_1, j_{2} \in \mathbb{N}$ such that 
$\frac{k_{j_1}}{n_{j_1}} \neq \frac{k_{j_2}}{n_{j_2}}$. 

We then define for $n \geq 0$  
\begin{displaymath}
O_n = F^{n_{j_1}} (O_{n -1}) \cup ( B(x; \varepsilon) + n(0, k_{j_1}) )
\end{displaymath}

Observe that $ F^{n_{j_1}} (O_0) \cap ( B(x; \varepsilon) + (0, k_{j_1}) ) \neq \emptyset$, so that $O_1$ is connected. 
Since $B(x; \varepsilon) + (0, k_{j_1}) \subseteq O_1$ and $F(x + p,y + q) = F(x,y) + (p,q)$ for all $(p,q) \in \mathbb{Z}^2$
we have that  $ F^{n_{j_1}} (O_1) \cap ( B(x; \varepsilon) + 2(0, k_{j_1}) ) \neq \emptyset$, so that $O_2$ is also connected. 
We see by induction that $O_n$ is connected for all $n \in \mathbb{N}$.

Define analogously $V_0 = B(x; \varepsilon)$ and for all $n \geq 0$ 
\begin{displaymath}
V_n = F^{n_{j_2}} (V_{n -1}) \cup ( B(x; \varepsilon) + n(0, k_{j_2}) )
\end{displaymath} 
Clearly, $V_n$ is connected for all $n \in \mathbb{N}$.

We want to see now that $O_{n_{j_2}} \cap V_{n_{j_1}} \neq \emptyset$. For that note that it follows from the definitions that  
\begin{displaymath}
O_{n_{j_2}} \supseteq F^{n_{j_1}} (O_{n_{j_2} - 1})  \supseteq F^{2n_{j_1}} (O_{n_{j_2} - 2}) \supseteq \ldots \supseteq F^{n_{j_2} n_{j_1}} (B(x; \varepsilon))
\end{displaymath}
and
\begin{displaymath}
V_{n_{j_1}} \supseteq F^{n_{j_2}} (V_{n_{j_1} - 1})  \supseteq F^{2n_{j_2}} (V_{n_{j_1} - 2}) \supseteq \ldots \supseteq F^{n_{j_1} n_{j_2}} (B(x; \varepsilon))
\end{displaymath}
We conclude that  $O_{n_{j_2}} \cup V_{n_{j_1}}$ is connected, 
so it contains the image of a curve $\gamma$ connecting $x + (0, n_{j_2}k_{j_1})$ 
and $x + (0, n_{j_1}k_{j_2})$. Since $k := | n_{j_2}k_{j_1} - n_{j_1}k_{j_2} | \neq 0$ the set  
$K = \cup_{i \in \mathbb{Z} } ([\gamma] + i(0, k))$ satisfies the proposition. 
\end{proof}

The proof of the next claim is similar to the proof of Proposition $9$ in \cite{AT1}.  
Some of the ideas that follow, especially those concerning figure $2$ below, can be traced 
back to the same paper.  

\begin{claim} 
$\overline{ \pi(\omega(B_{0})) } = \mathbb{T}^2$. 
\end{claim}

\begin{proof}
Assume by contradiction that there is $P \in \mathbb{T}^2$ and $\varepsilon > 0$ such that 
$B(P; \varepsilon) \cap \pi(\omega(B_0)) = \emptyset$.

By the previous proposition, there is a connected set $K \subset \mathbb{R}^2$ such that  
for all $y \in K$ there is $n(y) \leq 0$ such that $\pi(F^{n(y)} (y) ) \in B(P; \varepsilon)$. 
Furthermore $K$ is horizontally bounded and  vertically unbounded
(both upward and downward), so it separates the sets 
$R= \{ y \in \mathbb{R}^2 \, | \,  (y)_1 > M \}$ and 
$L = \{ y \in \mathbb{R}^2 \, | \, (y)_1 < - M \}$ for any $M$ greater than the horizontal bound for $K$. 

Take $x \in \omega(B_0)$ and let $\Gamma$ be the connected component of $\omega(B_{0})$ that passes through $x$. 
Take $p = M + \lfloor (x)_1 \rfloor + 1$. Since $p \in \mathbb{Z}$ the set $K + (p, 0)$ 
also satisfies Proposition $5$. Note that $x$ is leftward of $K + (p, 0)$ 
and recall that $\Gamma$ is unbounded rightward so we have that $\Gamma \cap (K+(0,p)) \neq \emptyset$.  

But if $z \in \Gamma \cap (K+(0,p))$ then $z \in \omega(B_{0})$ and there is 
$n(z) \leq 0$ such that $\pi( F^{n(z)} (z) ) \in B(P; \varepsilon)$. 
In particular, $F^{n(z)} (z) \in \omega(B_{0})$ (since $\omega(B_0)$ is completely invariant)
and $\pi(F^{n(z)} (z)) \in B(P; \varepsilon)$, a contradiction. 
\end{proof}

We can show analogously that $\overline{ \pi(\omega(B_{\pi })) } = \mathbb{T}^2 $.

Since $\overline{ \pi(\omega(B_{0})) } = \mathbb{T}^2 $ the set 
$\pi(\omega (B_0))$ must contain at least one non-fixed point and therefore 
$\omega (B_0)$ must contain a non-fixed point that we'll denote by $x$. 
By continuity there is a $\varepsilon > 0$ such that  
$F(B(x; \varepsilon)) \cap B(x; \varepsilon) = \emptyset$.  
Since we also have that $\overline{\pi(\omega(B_{\pi}))} = \mathbb{T}^2$, there is $(p,q) \in \mathbb{Z}^2$
and $y \in \mathbb{R}^2$ such that $\| (y + (p, q)) - x \| < \varepsilon$ 
and $y + (p, q) \in \omega(B_{\pi}) + (p,q)$. 

Since $\pi (\omega(B_{0})) \cap \pi(\omega(B_{\pi})) = \emptyset$ the distance between the 
compact sets $\omega(B_{0}) \cap \overline{B(x; \varepsilon)}$ and  
$(\omega(B_{\pi}) + (p, q)) \cap \overline{B(x; \varepsilon)}$ 
is a strictly positive number $d \leq 2 \varepsilon$ and is realized by points $x_1 \in \omega(B_0)$ 
and $y_1 \in (\omega(B_{\pi}) + (p,q))$. 

Let $v$ be the \emph{open} line segment connecting $x_1$ to $y_1$ (see Figure $2$). 
Observe that by our choice of $x_1$ and $y_1$ we have $F(v) \cap v = \emptyset$ and 
 $v \cap \omega(B_0) = \emptyset = v \cap (\omega(B_{\pi}) + (p, q))$.  
 Let $\Gamma$ be the connected component of $\omega(B_{0})$ that contains 
$x_1$ and let $\Theta$ be the connected component of 
$\omega(B_{\pi}) + (p,q)$ that contains $y_1$. 

We know from \cite{T1} (Proposition $8$) the following fact. 

\begin{figure}[t]
\includegraphics[scale=0.9]{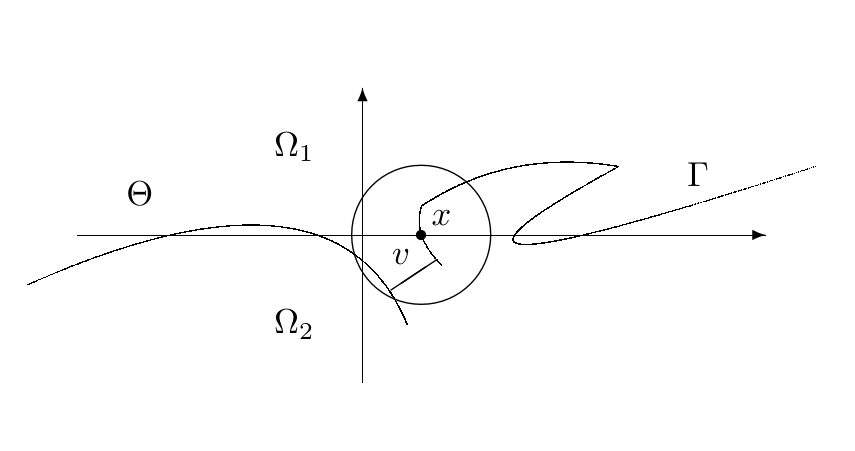}
\caption{$\Gamma$ denotes a connected component of $\omega(B_0)$ and 
$\Theta$ denotes a connected component of $\omega(B_\pi) + (p,q)$.   
The open line segment $v$ links $\Gamma$ and $\Theta$ in an $\varepsilon$-ball centered in $x \in \Gamma$. 
The set $(\Gamma \cup v \cup \Theta)^C$ 
has two connected components we call $\Omega_1$ and $\Omega_2$}
\end{figure}

\begin{claim} \label{unica cc}
The set 
\begin{displaymath}
A = ( \omega(B_{0}) \cup (\omega(B_{\pi}) + (p, q)) )^C
\end{displaymath}
has a single connected component. In particular, if $\Gamma$ is a connected component of $\omega(B_0)$ 
and $\Theta$ is a connected component of $(\omega(B_{\pi}) + (p, q))$ then
the set $(\Gamma \cup \Theta)^C$ has a single connected component. The set 
$(\Gamma \cup v \cup \Theta)^C$ has exactly two connected components.
\end{claim} 

The next claim uses an argument similar to the one used in Claim $2$.  

\begin{claim} \label{curva por todos} 
For all $y \in \mathbb{T}^2$ there is $z \in \pi^{-1} (y) \cap A$ and a continuous connected function $\delta:\mathbb{R}\to A$ such that
$\delta(i)=z+(0,i)$ for all integers $i$, and such that $\lim_{|t|\to\infty}||\delta(t)||=\infty$. Furthermore, the image of $\delta$ is a horizontally bounded set 
$[\delta]$ that
separates $\omega(B_{0})$ and $\omega(B_{\pi}) + (p, q)$. 
\end{claim}

\begin{proof} 
Since $\pi (\omega(B_{0})) \cap \pi(\omega(B_{\pi})) = \emptyset$ assume, without loss of 
generality, that $y \notin \pi (\omega(B_{0}))$. 
Then there is a point $z \in \pi^{-1}(y)$ such that $(z)_1 > p$. For $i \in \mathbb{Z}$ all points 
of the form $z + (0,i)$ 
have first coordinates strictly greater than $p$ and therefore none of them belongs to $\omega(B_{\pi}) + (p, q)$. 
In particular, $\{ z + (0,i) \; | \; i \in \mathbb{Z} \} \cap A^C = \emptyset$. 

By claim $\ref{unica cc}$ the set $A$ has a single connected component hence 
$z$ and $z + (0,1)$ are in the same connected component of $A$. 
Consider a curve $\gamma: [0,1] \to A$ such that 
$\gamma(0) = z$ and $\gamma(1) = z + (0,1)$.  
The function $\delta(s) := \gamma(s-\lfloor s\rfloor)+(0,\lfloor s\rfloor)$ has image $[\delta]=\cup_{i \in \mathbb{Z} } ([\gamma] + (0,i))$ and satisfies the claim.  
\end{proof}

Denote by $\Omega_1$ and $\Omega_2$ the two connected components of 
$(\Gamma \cup v \cup \Theta)^C$. 
Note that $\Omega_1$ and $\Omega_2$ are open
and that $\partial \Omega_1 = \partial \Omega_2 = \Gamma \cup v \cup \Theta$. 

\begin{claim}
Let $y \in \mathbb{T}^2$. Take $z \in \pi^{-1} (y)$ and $\delta$ given by the previous claim. Then there is $T>0$ and $k_1, k_2 \in \{0,1\}$ with $k_1 \neq k_2$ such that, for
all $t>T$, $\delta(t)\in \Omega_{k_1}$ and for all $t<-T,\, \delta(t)\in \Omega_{k_2}.$ In particular, $z + (0, j_1) \in \pi^{-1} (y) \cap \Omega_{k_1}$
for all $j_1 > |T|$ and $z + (0, j_2) \in \pi^{-1} (y) \cap \Omega_{k_2}$ for all $j_2 < - |T|$. 
\end{claim}

\begin{proof}
Since $v$ is bounded and $\lim_{|t|\to\infty}||\delta(t)||=\infty$, there exists $T \in \mathbb{R}$ such that  
$\delta(t) \cap v = \emptyset$ for all $|t| > T$.  Since $[\delta] \subset A$, this implies that for all $t>|T|,\, \delta(t)$ belongs to either $\Omega_1$ or $\Omega_2,$ and since $\delta$ is continuous
there exists $k_1, k_2 \in \{0,1\}$ such that, for
all $t>T$, $\delta(t)\in \Omega_{k_1}$ and for all $t<-T,\, \delta(t)\in \Omega_{k_2}.$

It remains to be shown that $k_1\not= k_2.$
Assume this is not true. Then there is a curve $\beta$ that connects $\delta(-2T)$ to $\delta(2T)$ 
without leaving $\Omega_{k_1}$. 
The set $\delta(]- \infty, -2T) \cup [\beta] \cup \delta(]2T, + \infty[)$ is connected, horizontally bounded and vertically unbounded (both upward and downward). 
By definition, this set does not intercept the set $\Gamma \cup v \cup \Theta$. But this is a contradiction since $\Gamma \cup v \cup \Theta$ is connected and unbounded both leftward and rightward. 
\end{proof}

\begin{claim}
$F (\Omega_i) \cap \Omega_i \neq \emptyset$ for $i = \{1,2\}$. 
\end{claim}

\begin{proof}
Since $(0, 0) \in \rho(F)$ by Lemma $1$ 
there is $y \in \mathbb{T}^2$ fixed for $f$ and such that $F(z) - z = (0, 0)$
for all $z \in \pi^{-1}(y)$. 
We conclude therefore that both $\Omega_1$ and $\Omega_2$ have fixed points, which proves the desired result. 
\end{proof}

We can now prove the following proposition. 

\begin{claim}
Either 
$F(\overline{\Omega_1}) \subseteq \overline{\Omega_1}$ or 
$F(\overline{\Omega_2}) \subseteq \overline{\Omega_2}$. 
\end{claim}

\begin{proof}
By the definition of $v$ we have that $v \cap \omega(B_0) = \emptyset = v \cap (\omega(B_{\pi}) + (p, q))$.  
Since $\omega(B_0)$ and $\omega(B_\pi)$ are invariant, $F(v) \cap \Gamma = \emptyset = F(v) \cap \Theta$.  
Furthermore, $v$ was chosen such that $F(v) \cap v = \emptyset$. Hence, 
\begin{displaymath}
F(v) \cap (\Gamma \cup v \cup \Theta) = \emptyset
\end{displaymath}
so $F(v)$ is either in $\Omega_1$ or in $\Omega_2$. 
Let's assume, without loss of generality, that $F(v) \subseteq \Omega_1$.
Observe that , since $\Gamma$ is a connected component of $\omega(B_{\pi})$ and $\omega(B_{\pi})$ is totally invariant,  
$F(\Gamma)$ is a connected component of $\omega(B_{\pi })$. 
Therefore, either $F(\Gamma) = \Gamma$ or $F(\Gamma) \cap \Gamma = \emptyset$. 

In the first case  we have that $F(\Gamma) \cap \Omega_2 = \Gamma \cap \Omega_2 = \emptyset$. 
This is also true in the second case: It's clear that 
$F(\Gamma) \cap \partial\Omega_1 = \emptyset$ since by the definition of $v$ we have $F(\Gamma) \cap v = \emptyset$ and 
$\omega(B_0) \cap \omega(B_{\pi}) = \emptyset$ implies $F(\Gamma) \cap \Theta = \emptyset$. 
Since $\Gamma \cup v $ is connected 
$F(\Gamma \cup v)$ is also connected so $F(v) \subseteq \Omega_1$ implies $F(\Gamma) \cap \Omega_2 = \emptyset$. 
We see analogously that $F(\Theta) \cap \Omega_2 = \emptyset$. 
We note that in any case $F(\partial \Omega_1) \cap \partial \Omega_1 = \emptyset$. 

Since $F(\partial \Omega_1) \cap \partial \Omega_1 = \emptyset$ and 
$F(v) \subseteq \Omega_1$ we have that $F(\partial \Omega_1) \cap \Omega_2 = \emptyset$. 
But $\Omega_1$ and $\Omega_2$ are connected 
so either $\Omega_2 \subseteq F(\Omega_1)$ or $\Omega_2 \cap F(\Omega_1) = \emptyset$. 

Assume that $\Omega_2 \subseteq F(\Omega_1)$: 
Then $F^{-1}(\Omega_2) \subseteq \Omega_1$, so $F^{-1}(\Omega_2) \cap \Omega_2 = \emptyset$. 
But by Claim $7$ we know $F(\Omega_2) \cap \Omega_2 \neq \emptyset$, a contradiction.  
Therefore we necessarily have that $\Omega_2 \cap F(\Omega_1) = \emptyset$ and since 
$F(\partial \Omega_1) \cap \partial \Omega_1 = \emptyset$ we necessarily must have
$F(\overline{\Omega_1}) \subseteq \overline{\Omega_1}$. 
\end{proof}

We are now ready to finish the proof of Theorem $2$. There is a small technicality to deal with, 
however the idea is simple and is illustrated by Figure $3$ bellow. 

\begin{figure}[h]
\includegraphics[scale=0.75]{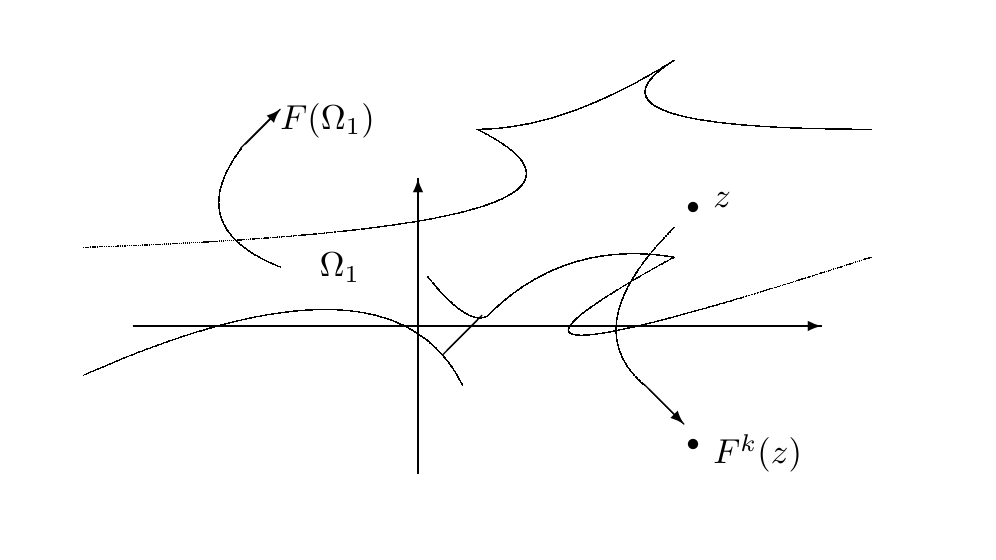}
\caption{We can't have $F(\overline{\Omega_1}) \subseteq \overline{\Omega_1}$ since by \cite{F1} 
there is a point $z \in \mathbb{R}^2$ such that $F^k(z) - z = k(0, 1)$. A similar argument show that 
the case $F(\overline{\Omega_2}) \subseteq \overline{\Omega_2}$ also can't happen}
\end{figure}


\begin{proof} [Proof of Theorem $2$]
Assume that 
 $F(\overline{\Omega_1}) \subseteq \overline{\Omega_1}$ (the other case is analogous). 
Then by induction $F^s(\overline{\Omega_1}) \subseteq \overline{\Omega_1}$ for 
$s \in \mathbb{N}$. 
Since $(0, 1), (0, -1) \in \rho(F)$ we know 
by Lemma $1$ that there are $y_1, y_2 \in \mathbb{T}^2$ both fixed for $f$ and such that  for 
all $k \in \mathbb{Z}$ we have 
\begin{displaymath}
F^k(z) - z = k(0, 1) \qquad \text{and} \qquad F^k(w) - w = k(0, -1) 
\end{displaymath}
for all $z \in \pi^{-1}(y_1)$ and $w \in \pi^{-1}(y_2)$. 

We can now use Claim $6$ putting $y = y_1$ to find 
$z \in \pi^{-1}(y_1)$ and a function $\delta_1$ such that $[\delta_1]$
that separates $\omega(B_0)$ and $\omega(B_{\pi}) + (p,q)$. From Claim $7$ we get, without loss of 
generality, that  $z \in \Omega_1$. 
We repeat the same procedure putting $y = y_2$ in Claim $6$ to find $\delta_2$ as before and, by Claim $7$, 
we get, without loss of generality, that  $w \in \Omega_1$. 

Note that if there is $k > 0$ such that either $F^k (z) = z + k(0, 1)$ or $F^k (w) = w +k (0, -1)$ belongs to $\Omega_2$
we would have that $ F^{k}(\overline{\Omega_1}) \cap \overline{\Omega_2}\not=\emptyset$, 
a contradiction since $F^{s}(\overline{\Omega_1}) \subseteq \overline{\Omega_1}.$ Therefore, for all $k\ge 0$, both $z+(0,k)$ and $w-(0,k)$ belong to $\Omega_1$, and thus, by Claim $7$, there exists $T_1>0$ such that,
if $k<-T_1$, then $z+(0,k)\in\Omega_2$.

Let $\beta:[0,1]\to A$ be a curve joining $z$ and $w$. There exists $L>T_1$ such that, if $|k|>L$ then $[\beta]+(0,k)$ is disjoint from $v$, and therefore $[\beta]+(0,k)$ belongs to $\Omega_1$ if $k>L$, and to $\Omega_2$ if $k<-L$. But this implies that 
$w-(0,k)\in\Omega_2$ for $k<-L$, our final contradiction.  
\end{proof}


 

\begin{thebibliography}{1}

\bibitem{AT1} S. Addas-Zanata and F. A. Tal, Homeomorphisms of the annulus 
with a transitive lift, Math. Z., v. 267, p. 971-980, 2011.

\bibitem{AT2} S. Addas-Zanata and F. A. Tal,  On generic rotationless diffeomorphisms of the annulus, 
Proc. Am. Math. Soc., v. 138, p. 1023, 2010. 

\bibitem{ATG} S. Addas-Zanata, F.A. Tal and B.A. Garcia, 
Dynamics of homeomorphisms of the torus homotopic to Dehn twists,  
(Arxiv preprint arXiv:1111.5561v1), 2011. 

\bibitem{AP} Alpern, S. and Prasad, V. S., Topological ergodic theory and mean rotation,
Proc. Am. Math. Soc. v. 118 (1993), 279-284.

\bibitem{A1} G. Atkinson, Recurrence of co-cycles and random walks, 
J. London Math Soc. \textbf{2}, 13 (1976), 486--488. 

\bibitem{D1} P. D\'{a}valos, On torus homeomorphisms whose rotation set is an interval,  
(Arxiv preprint arXiv:1111.2378v1), 2011. 
  
\bibitem{F1} J. Franks, The rotation set and periodic points for torus 	homeomorphisms, ``Dynamical systems and chaos.", 
Aoki, Shiraiwa and Takahashi ed. World Scientific, Singapore, (1995), 41--48. 

\bibitem{F2} J. Franks, Realizing rotation vectors for torus homeomorphisms,  
Trans. Amer. Math. Soc. \textbf{311} (1989), no. 1, 107--115.

\bibitem{F3} J. Franks, Recurrence and fixed points for surface homeomorphisms,  
Ergod. Th. Dynam. Sys. \textbf{8} (1988), 99--107. 

\bibitem{KT1} A. Koropecki and F. A. Tal, Strictly toral dynamics, (Arxiv preprint arXiv:1201.1168v2 ), 2012.

\bibitem{KT2} A. Koropecki and F. A. Tal, Area-preserving irrotational diffeomorphisms of 
the torus with sublinear diffusion (Arxiv preprint arXiv:1206.2409v1), 2012. 

\bibitem{KK} A. Kocsard and A. Koropecki,  
Free curves and periodic points for torus homeomorphisms,  
Erg. Th. \& Dyn. Sys. (2008), 28, pp. 1895-1915

\bibitem{LM} J. Llibre and R. Mackay, Rotation vectors and entropy for 
homeomorphisms of the torus isotopic to the identity, 
Erg. Th. \& Dyn. Sys. \textbf{11} (1991) 115--128. 

\bibitem{MZ} M. Misiurewicz and K. Ziemian, Rotation sets for maps of Tori,   
J. London Math. Soc. (2) \textbf{40} (1989) 490--506. 

\bibitem{T1} F. A. Tal, Transitivity and rotation sets with nonempty interior for 
             homeomorphisms of the 2-torus, Proc. Am. Math. Soc. v. 140 (2012), 3567-3579.
  
\end{thebibliography}
\end {document}